\documentclass[12pt,oneside]{amsart}
\usepackage{amsmath}
\usepackage{amsthm}
\usepackage{amsfonts}
\usepackage{amssymb}
\usepackage{enumerate}
\newtheorem{theorem}{Theorem}
\newtheorem{lemma}[theorem]{Lemma}
\newtheorem{proposition}[theorem]{Proposition}

\theoremstyle{definition}
\newtheorem{definition}[theorem]{Definition}

\theoremstyle{remark}
\newtheorem{remark}[theorem]{Remark}
\newcommand{\DD}{{\mathbb D}}
\newcommand{\OO}{{\mathcal O}}

\newcommand{\PP}{{\mathcal P}}

\newcommand{\GG}{{\mathbb G}}

\newcommand{\NN}{{\mathbb N}}

\newcommand{\CC}{{\mathbb C}}
\newcommand{\cC}{{\mathcal C}}

\newcommand{\TT}{{\mathbb T}}

\DeclareMathOperator{\Aut}{Aut} \DeclareMathOperator{\id}{id}

\DeclareMathOperator{\tr}{tr}
\DeclareMathOperator{\Int}{int}

\renewcommand{\phi}{\varphi}


\begin{document}

\title{The group of automorphisms of the~pentablock}

\address{Universit\'e Laval, QC, Canada}
\author{\L ukasz Kosi\'nski}\email{lukasz.kosinski@im.uj.edu.pl}

\thanks{The author is partially supported by the Iuventus Plus grant}

\maketitle

\begin{abstract} Answering questions posed in \cite{Agl-Lyk-You14} we determine the group of holomorphic automorphisms of the pentablock. The geometry of the pentablock (linear convexity, being an analytic retract of special convex domains) is also studied.\end{abstract}

\section{Introduction}

The pentablock, a domain defined in \cite{Agl-Lyk-You14}, arises from a test example corresponding to a single uncertainty structure in the $2 \times 2$ case. It is a subdomain of $\CC^3$ denoted by $\PP$ and defined as the image of the classical $2\times 2$ Cartan domain of the first type $\mathcal R_I=\{z\in\CC^{2\times 2}:\ ||z||<1\}$ under the mapping $$\pi:z=(z_{ij})\mapsto (z_{12}, \tr z, \det z).$$

The main aim of this note is to determine the group of holomorphic automorphisms of $\mathcal P$ (see Theorem~\ref{autp}). It turns out that $\Aut(\PP)$ is not transitive, moreover the orbit of $0$ forms an 1-dimensional variety. We shall also investigate geometric properties of the pentablock. In particular, we shall prove that the pentablock is not an analytic retract of the open unit ball of any $\mathcal J^*$ algebra of finite rank (see Proposition~\ref{retrp}). These results answer questions posed in \cite{Agl-Lyk-You14}.

\bigskip

The paper is organized as follows. We start with describing geometric properties of the pentablock which will be crucial for our considerations. In particular, we shall recall different characterizations of points in $\PP$ due to Agler, Lykova and Young, present some actions playing on the pentablock and express $\PP$ as a Hartogs domain over the special domain. Next we shall focus on its convexity properties which will turn out to be quite useful in the sequel. The structure of the boundary of $\PP$ will also be studied.

Then, in the next sections, making use of obtained properties of the pentablock we shall prove our main results.

\bigskip

Here is some notation. Throughout the paper $\mathbb D$ denotes the unit disc in the complex plane, additionally by $\mathbb T$ we shall denote the unit circle. $\partial_s D$ stands for the Shilov boundary with respect to the algebra $\mathcal O(D)\cap \mathcal C(\bar D)$ of a bounded domain $D$ in $\CC^n$.

\section{Geometric properties of the pentablock}

\subsection{Characterizations of points in $\mathcal P$} One of the main results of \cite{Agl-Lyk-You14} contains several descriptions of the pentablock. Since we will exploit almost all of them, we recall them for the convenience of the Reader:
\begin{theorem}[\cite{Agl-Lyk-You14}, Theorem~1.1, Theorem~5.2]\label{th:severaldescriptions}
Let 
$$(s,p)=(\lambda_1 + \lambda_2, \lambda_1 \lambda_2),$$
where $\lambda_1,\lambda_2\in \DD$. Let $a\in \mathbb C$ and $$\beta=\frac{s- \bar s p}{1- |p|^2}.$$ The following statements are equivalent:
\begin{enumerate}
\item[(1)] $(a,s,p)\in \mathcal P$,
\item[(2)] $|a|< \left|1- \frac{\frac12 s\bar \beta}{1+ \sqrt{1-|\beta|^2}}\right|$,
\item[(3)] $|a|< \frac12 |1- \bar \lambda_2 \lambda_1| + \frac12 (1-|\lambda_1|^2)^\frac12 (1-|\lambda_2|^2)^\frac12,$
\item[(4)] $\sup_{z\in \DD} |\Psi_z(a,s,p)|<1$, where $\Psi_z$ is the linear fractional map $$\Psi_z(a,s,p)=\frac{a(1-|z|^2)}{1-sz + pz^2}.$$
\end{enumerate}
\end{theorem}

\subsection{Actions on the pentablock}\label{sec:act} It is clear that $\left( \begin{array}{cc} \lambda z_{11} & z_{12} \\ \lambda^2 z_{21} & \lambda z_{22} \end{array} \right)\in \mathcal R_I$ whenever $\lambda\in \DD$ and $\left( \begin{array}{cc} z_{11} & z_{12} \\ z_{21} & z_{22} \end{array} \right)\in \mathcal R_I$. Therefore, since $\mathcal P=\pi(\mathcal R_I)$, we find that the pentablock is a $(0,1,2)$-balanced domain, which means that $(a, \lambda s, \lambda^2 p)\in \mathcal P$ for any $(a,s,p)\in \mathcal P$ and $\lambda \in \mathbb D$.

Moreover, thanks to Theorem~\ref{th:severaldescriptions} one can observe that the pentablock is a Hartogs domain in $\mathbb C^3$ over a domain $\mathbb G_2:=\{(\lambda_1 + \lambda_2, \lambda_1 \lambda_2):\ \lambda_1, \lambda_2\in \DD\}$. More precisely,
\begin{equation}\label{penta} \mathcal P=\left\{(a,s,p)\in \CC \times \GG_2:\ |a|< \left|1- \frac{\frac12 s\bar \beta}{1+ \sqrt{1-|\beta|^2}}\right| \right \},
\end{equation}
where $\beta = \frac{s-\bar s p}{1-|p|^2}$.

An immediate consequence of the description mentioned above is the fact that $\mathcal P$ is a $(1,0,0)$-balanced domain i.e. $(\lambda a, s, p)\in \PP$ for any $(a,s,p)\in \PP$ and $\lambda\in \DD$. 

Therefore, joining facts mentioned above we see that $\mathcal P$ is $(k,1,2)$-balanced for any $k\in \mathbb N$.

Recall here that the domain $\GG_2$ appearing in \eqref{penta} is called the \emph{symmetrised bidisc} and that it may be expressed in coordinates $(s,p)$ in the following way: $$\mathbb G_2=\{(s,p)\in \CC^2 :\ |s- \bar s p| + |p|^2<1\}.$$

\subsection{Pentablock as a Hartogs domain and its convexity properties}

For the convenience of the Reader we start with recalling some notions of convexity and their properties (see also \cite{And-Pas2004}):
\begin{itemize} 
 \item a bounded domain $D$ of $\mathbb C^n$ is called \emph{$\mathbb C$-convex} if its intersection with any complex line is connected and simply-connected (i.e., according to the Riemann mapping theorem, it is a disc in an analytic sense). 
\end{itemize}
The definition implies that all $\mathbb C$-convex domains in the complex plane are just domains which are conformally equivalent with the unit disc. However, there is no such nice characterization of $\mathbb C$-convexity in higher dimensions just to mention the symmetrised bidisc which is $\mathbb C$-convex but not equivalent with any convex domain (\cite{Cos 2004b}).

We shall also deal with linear convexity. It is a very standard and simple fact that a domain $D\subset \CC^n$ is convex if and only if for any $a\not\in D$ there is a real hyperplane passing through $a$ and omitting $D$. Linear convexity may be seen as a complex counterpart of this property:
\begin{itemize} 
\item a set $D$ of $\mathbb C^n$ is said to be \emph{linearly convex} if for any $a\not\in D$ there is a complex hyperplane passing through $a$ and omitting $D.$
\end{itemize}
An immediate consequence of this definition is that any convex domain is linearly convex and that all domains in the complex plane are linearly convex. It is non-trivial that $\mathbb C$-convexity implies linear convexity (see \cite[Theorem~2.3.9]{And-Pas2004}), so in view of the situation in the complex plane we see that $\CC$-convexity is a stronger notion.

Both $\CC$-convexity and linear convexity imply many nice properties. For instance, any linearly convex domain is a domain of holomorphy, what is more, under some additional conditions it is even polynomially convex. $\CC$-convexity implies hyperconvexity.

As already mentioned, the pentablock is a Hartogs domain in $\CC^3$. To author's knowledge the are hardly two papers devoted to $\CC$-convexity (linear convexity) of Hartogs domains (see \cite{Kis}, \cite{Jac 2006}), both of them containing some results for very special class of domains: Hartogs domains in $\CC^2$ whose base is the unit disc. The results that appeared in the first mentioned paper are also gathered in \cite{And-Pas2004} and some of them may be find generalized in \cite{Jac}.

For the convenience of the Reader we shall recall in Theorem~\ref{th:hart} the result describing $\mathbb C$-convexity of Hartogs domains in $\CC^2$ over $\DD$ in terms of their defining functions.

To do it we need the following
\begin{definition}A real valued $\cC^2$-smooth function $u$ defined on a subdomain $D$ of $\CC$ is said to be \emph{$\CC$-convex} if the inequality 
\begin{equation}\label{df:c-con} u_{z\bar z}\geq |u_{zz} - (u_z)^2|^2\end{equation} holds on $D$. Analogously, a real valued $\cC^2$ function $u$ of several complex variables is \emph{$\CC$-convex} if its restriction to any complex line is $\CC$-convex in the previous sense.
\end{definition}

Direct calculations show that $u\circ T$ is $\CC$-convex for any complex affine mapping $T$ and a $\CC$-convex function $u$.
\begin{remark}\label{rem:conf}
It is probably well known that pluriharmonic and $\CC$-convex functions on $D\subset \CC^n$ are of the form $u(z)=-2\log |\sum_{j=0}^n \alpha_j z^j|$, where $\alpha_j\in \CC$ are such that $z\mapsto \sum_{j=0}^n \alpha_j z^j$ does not vanish on $D$. 

To prove this property recall that any pluriharmonic function $u$ may be written locally as $u(z)= - \log f(z) \overline{f(z)}$, where $f$ is holomorphic. Then it is enough to observe that condition \eqref{df:c-con} expressed in terms of $f$ means that the second derivative of $f$ vanishes.
\end{remark}

\begin{theorem}[\cite{And-Pas2004} Theorem~2.5.16, Example~2.5.16, Proposition~2.5.9]\label{th:hart} Let $u\in \mathcal C^2(\DD)$ be a real function. Then the Hartogs domain $E=\{(w,z)\in \CC\times \DD:\ |w|^2< e^{-u(z)}\}$ is $\CC$-convex if and only if it is linearly convex if and only if $u$ is $\CC$-convex.
\end{theorem}

\begin{remark}\label{rem:DtimesG}
Note that there is no direct counterpart of the above result for Hartogs domains over $\mathbb C$-convex domains, for example $\GG_2\times \DD$ is not $\CC$-convex (for the $\CC$-convexity of the symmetrised bidisc see \cite[Theorem~1]{Nik-Pfl-Zwo 2008}).
\end{remark}

In view of Theorem~\ref{th:hart} it is natural to express \eqref{penta} as 
\begin{equation}\label{penta1} \mathcal P=\{(a,s,p)\in \mathbb C\times \mathbb G_2:\ |a|^2<e^{-\varphi(s,p)}\},
\end{equation}
 where 
\begin{equation}\label{pentaphi}\varphi(s,p)= -2\log \left|1- \frac{\frac12 s\bar \beta}{1+ \sqrt{1-|\beta|^2}}\right|,\quad (s,p)\in \GG_2.
\end{equation}
Observe that $\varphi$ is smooth on $\GG_2$. Moreover, we have:
\begin{proposition}\label{con1} The function $\varphi$ is $\CC$-convex.
\end{proposition}

\begin{proof}[Proof of Proposition~\ref{con1}] 

First note that it follows from Theorem~\ref{th:severaldescriptions} (4) that the pentablock may be given as
\begin{equation}\label{inter}\PP=\Int \bigcap_{z\in \DD}\{(a,s,p)\in \CC\times \GG_2:\ |\Psi_z(a,s,p)|< 1\}.
\end{equation}

Take any complex line $l$ intersecting $\GG_2$ and any (one-dimensional) disc $\Delta$ contained in $l \cap \GG_2$. We aim at showing that $u$ restricted to $\Delta$ is $\CC$-convex. Thanks to Theorem~\ref{th:hart}, it suffices to show that the following Hartogs domain $\{(a,s,p)\in \CC\times \Delta:\ |a|^2< e^{-u(s,p)}\}$ is linearly convex in $\CC\times l$.

To do it note that $\PP_z:=\{(a,s,p)\in \CC\times \GG_2:\ |\Psi_z(a,s,p)|<1\}$ is a Hartogs domain, more precisely $$\PP_z=\{(a,s,p)\in \CC\times \GG_2: |a|^2< e^{- \varphi^z(s,p)}\},$$ where $\varphi^z(s,p)=-2\log\left| \frac{1-sz + pz^2}{1-|z|^2} \right|,$ $(s,p)\in \GG_2$, $z\in \DD$. One can check that $\varphi^z$ is $\CC$-convex (see also Remark~\ref{rem:conf}), so $\PP_z\cap (\CC\times \Delta)$ is a $\CC$-convex subdomain of $\CC\times l$, by Theorem~\ref{th:hart}. Therefore, $\bigcap_{z\in \DD} (\PP_z\cap (\CC\times l))$ is linearly convex. Since the interior of a linearly convex set is linearly convex,  $\PP\cap (\CC\times \Delta)$ is linearly convex, as well.

\end{proof}

Note that in the proof of the previous proposition we have actually obtained the following result which is well known for plurisubharmonic and convex functions:
\begin{proposition} Let $\{u_\alpha\}$ be a locally bounded from above family of $\CC$-convex functions on a domain $D$ of $\CC^n$. Let $u:=\sup\{u_\alpha\}$. If $u$ is $\cC^1$ smooth, then it is $\CC$-convex.
\end{proposition}

\begin{remark}\label{rem:phi}
Note that it follows from Proposition~\ref{con1} and Remark~\ref{rem:conf} that the function $\varphi$ given by \eqref{pentaphi} is not pluriharmonic.
\end{remark}

Finally, we shall show that $\PP$ is linearly convex. Note that this result implies immediately number of properties of $\PP$ such as polynomial convexity (for a direct proof of polynomial convexity of the pentablock see \cite[Theorem~6.3]{Agl-Lyk-You14}). It would be also interesting to know whether the pentablock is $\CC$-convex. As noted in Remark~\ref{rem:DtimesG} this fact cannot be deduced using directly methods working for Theorem~\ref{th:hart}. The question about $\CC$-convexity of $\PP$ and the next result (Proposition~\ref{prop:lin}) are very interesting in the light of Lempert's theorem on the equality of holomorphically invariant functions and metrics for the pentablock. Let us mention here that it is a long-standing open problem whether Lempert's theorem holds for $\CC$-convex domains, so far it is known for smooth $\CC$-convex domains (see \cite{Lem 1982}).
\begin{proposition}\label{prop:lin}
The pentablock is linearly convex.
\end{proposition}
\begin{proof}
Let $(a_0,s_0,p_0)\not\in  \PP$. We are looking for a complex hyperplane passing through $(a_0,s_0,p_0)$ and omitting $\mathcal P$. 

Assume first that $(s_0,p_0)\not\in \GG_2$. Since $\GG_2$ is $\CC$-convex (see \cite[Theorem~1]{Nik-Pfl-Zwo 2008}), it is linearly convex, so there is complex line $l$ passing through $(s_0,p_0)$ and omitting $\GG_2$. Then $\CC\times l$ is a complex hyperplane we are looking for.

Now suppose that $(s_0,p_0)\in \GG_2$. Applying Theorem~\ref{th:severaldescriptions} we get that there is $z\in \overline \DD$ and $\omega\in\CC\setminus \DD$ such that  $\Psi_{z_0}(a_0, s_0, p_0)=\omega$ and $\Psi_{z_0}(a,s,p)\neq \omega$ for all $(a,s,p)\in \PP$. Thus $\{(a,s,p)\in \CC^3:\ \Psi_z(a,s,p)=\omega\}$ describes a hyperplane satisfying the desired properties.
\end{proof}

\subsection{Geometry of the symmetrised bidisc}
Recall that the symmetrised bidisc is defined as $$\GG_2=\{(\lambda_1 + \lambda_2, \lambda_1 \lambda_2):\ \lambda_1,\lambda_2\in \DD\}.$$ Its Shilov boundary $\partial_s \GG_2$ is given by $\partial_s \GG_2=\{(\lambda_1 + \lambda_2, \lambda_1 \lambda_2):\ \lambda_1, \lambda_2\in \TT\}.$ 

Let $\Sigma$ denote the \emph{royal variety} of the symmetrised bidisc, i.e. $\Sigma=\{(2\lambda, \lambda^2):\ \lambda\in \DD\}$.

In the sequel we will make use of the description of the set of proper holomorphic self-mappings of the symmetrised bidisc due to Edigarian and Zwonek:
\begin{theorem}[\cite{Edi-Zwo 2005}, Theorem~1]\label{Edi-Zwo}
Let $f:\GG_2\to \GG_2$ be proper and holomorphic. Then there is a finite Blaschke product $b$ such that $$f(\lambda_1+\lambda_2, \lambda_1 \lambda_2) = (b(\lambda_1) + b(\lambda_2), b(\lambda_1) b(\lambda_2)),\quad \lambda_1, \lambda_2\in \DD.$$
\end{theorem}

\subsection{Geometric properties of the topological boundary of the pentablock}

We shall divide the boundary of the pentablock on the three parts $\partial \mathcal P= \partial_1 \mathcal P \cup \partial_2 \mathcal P\cup \partial_3 \mathcal P$, where $\partial_1 \mathcal P= \partial \mathcal P\cap \mathbb C\times \GG_2$, $\partial_2 \mathcal P = \partial \mathcal P \cap (\CC\times  (\partial \GG_2\setminus \partial_s \GG_2))$ and $\partial_3 \mathcal P = \partial \mathcal P \setminus (\partial_1 \PP \cup \partial_2 \PP)$. Let us list few properties of them:

1) It is clear that $$\partial_1\PP=\{(a,s,p)\in \CC \times \mathbb G_2:\ |a|^2=e^{-\varphi(s,p)}\}.$$ Topological codimension of $\partial_1 \PP$ is equal to $1$ and any point $x$ of $\partial_1 \PP$ is a smooth point of $\partial \PP$, what is more $r(a,s,p)=\log |a|^2 +\varphi(s,p)$ is a local defining function of $\partial \PP$ in a neighborhood of $x$. We shall show
\begin{lemma}\label{lem:par1} $\partial_1 \PP$ is not Levi flat (i.e. the Levi form of the defining function $r$ restricted to the complex tangent space is not degenerate on $\partial_1 \PP$). Moreover $\partial_1 \PP$ may be foliated with analytic discs.
\end{lemma}
\begin{remark}
The Reader not familiar with Levi forms may read the above lemma in the following way:  there are no 2-dimensional analytic discs in $\partial_1 \PP$ and $\partial_1\PP$ may be foliated with 1-dimensional analytic discs.
\end{remark}

\begin{proof}[Proof of Lemma~\ref{lem:par1}] Take $x\in \partial_1\PP$ and $z\in \partial \mathcal R_I$ such that $\pi(z)=x$. By the singular value decomposition theorem, there are unitary matrices $U$ and $V$ and $\zeta_0\in \overline\DD$ such that $z=U\left(\begin{array}{cc} 1 & 0\\ 0 & \zeta_0 \end{array}\right)V$. Since $x\in \overline \DD\times \GG_2$, $\zeta_0\in \DD$. Then $t\mapsto \pi(U\left(\begin{array}{cc} 1 & 0\\ 0 & t \end{array}\right)V)$ is a non-trivial analytic disc in $\bar \PP$ passing through $x$. Since $\varphi$ is plurisubharmonic (see e.g. Proposition~\ref{con1}), we find that this analytic disc lies in $\partial \PP$. 

To get the desired assertion it suffices to show that the Levi form of $r$ is not equal to $0$. Of course this fact may be checked directly. But to avoid tedious computations it suffices to note that otherwise the function $\varphi$ would be pluriharmonic. But this contradicts Remark~\ref{rem:phi}.
\end{proof}

2) Clearly 
\begin{multline}\nonumber \partial_2 \PP=\{ (a,\lambda_1 + \lambda_2, \lambda_1 \lambda_2)\in \CC \times \bar\GG_2:\  (\lambda_1,\lambda_2)\in (\DD\times \TT) \cup (\TT \times \DD), \\ |a|< \frac{1}{2} |1- \bar \lambda_2 \lambda_1| + \frac{1}{2} (1-|\lambda_1|^2)^{\frac{1}{2}} (1-|\lambda_2|^2)^{\frac{1}{2}}\}.
\end{multline}
Of course, the topological codimension of $\partial_2 \PP$ is equal to $1$.
\begin{lemma}\label{lem:par2}
$\partial_2\PP$ is a Levi flat part of $\partial \PP$.
\end{lemma}
\begin{remark}
The above lemma may be stated without using the notion of Levi forms in the following way: $\partial_2 \PP$ may be foliated with 2-dimensional analytic discs.
\end{remark}

\begin{proof}[Proof of Lemma~\ref{lem:par2}] Let $x=(a_0, s_0, p_0)=(a_0, \lambda_1 + \lambda_2, \lambda_1 \lambda_2 ) \in \partial_2 \PP$. Losing no generality we may assume that $|\lambda_1|=1$ and $|\lambda_2|<1$. Then $(a,\lambda)\mapsto (a, \lambda_1 + \lambda, \lambda_1 \lambda)$ defined in a neighborhood of $(a_0, \lambda_2)$ lies in $\partial_2 \PP$.
\end{proof}

3) Topological codimension of $\partial_3 \PP$ is equal to 2.

\section{Group of automorphisms of $\mathcal P$}

Recall that the special subgroup of the group of automorphisms of the pentablock was constructed in \cite{Agl-Lyk-You14}. More precisely, it was shown that any mapping of the form
\begin{multline}\label{con} f_{\omega, \nu}(a, \lambda_1 + \lambda_2, \lambda_1 \lambda_2) =\\ \left(\frac{\omega \eta (1-|\alpha|^2) a}{1-\bar \alpha (\lambda_1 + \lambda_2) +\bar \alpha^2\lambda_1 \lambda_2}, \nu(\lambda_1)+\nu(\lambda_2), \nu(\lambda_1) \nu(\lambda_2) \right),\end{multline}
where $\omega\in \TT$ and $\nu$ is a M\"obius function of the form $\nu(\lambda) = \eta \frac{\lambda - \alpha}{1 - \bar \alpha \lambda},$ $\lambda\in \DD$, is an automorphism of the pentablock.

In what follows we shall show that above mappings form the whole group of automorphisms of the pentablock:
\begin{theorem}\label{autp} $$\Aut(\PP)= \{f_{\omega,\nu}:\ \omega\in \mathbb T,\ \nu\text{ is a M\"obius function}\}.$$ In particular, $\{0\}\times \Sigma$ is an orbit of $0$ and, consequently, the pentablock is inhomogeneous.
\end{theorem}

The crucial tools used in deducing the result presented above involve elementary properties of biholomorphic mappings between quasi-balanced domains, the description of proper holomorphic self-mappings of the symmetrised bidisc and the classical Cartan theorem.

We start with the following result:
\begin{lemma}
Any automorphism of the pentablock extends to a biholomorphic mapping between some neighborhoods of $\overline{\PP}$.
\end{lemma}

\begin{proof}
Recall (see \cite{Kos 2009}, Lemma~6 and Remark~7) that any proper holomorphic mapping between quasi-balanced domains whose Minkowski functionals are continuous extends holomorphically past neighborhoods of closure of domains. Note that for any $r\in (0,1)$ domains $\mathcal P_r:=\{(ra, rs, r^2p):\ (a,s,p)\in \mathcal P\}$ are relatively compact in $\mathcal P$, so the pentablock satisfies the assumption of \cite[Lemma~6]{Kos 2009}. Therefore, since any automorphism is trivially a proper map, applying \cite[Lemma~6]{Kos 2009} to an automorphism $\varphi$ of the pentablock and to its inverse $\phi^{-1}$ we find that both $\varphi$ and $\varphi^{-1}$ extend holomorphically past $\bar{\mathcal P}$. Moreover, the equality $\varphi\circ \varphi^{-1}=\id$ remains true in a neighborhood of $\bar\PP$, so the extension of $\varphi$ to $\bar\PP$ is a biholomorphic mapping between neighborhoods of $\bar\PP$.
\end{proof}

Now we are ready to find a description of the group of automorphism of the pentablock.

\begin{proof}[Proof of Theorem~\ref{autp}] Let us take $\varphi\in \Aut(\PP)$.

Step~1. Consider the mapping $$\Psi:\mathbb G_2\ni (s,p)\mapsto (\varphi_2(0,s,p), \varphi_3(0,s,p))\in \mathbb G_2.$$ As mentioned above, $\Psi$ extends holomorphically past $\bar\GG_2$. Moreover, any point $(0,s,p)$, where $(s,p)\in \partial \GG_2\setminus \partial_s \GG_2$, lies in a Levi flat part of the boundary of the pentablock, i.e. $(0,s,p)\in \partial_2 \PP$. Thus, $\varphi(0,s,p)$ lies generically in a Levi flat part, as well. Therefore, a simple continuity argument shows that $\Psi$ maps $\partial \GG_2$ into $\partial \GG_2$. In particular, $\Psi$ is a proper holomorphic self-mapping of $\GG_2$.

Step~2. It follows from Theorem~\ref{Edi-Zwo} that there is a Blaschke product $b$ such that $$\Psi(\lambda_1+\lambda_2, \lambda_1 \lambda_2) = (b(\lambda_1) + b(\lambda_2), b(\lambda_1) b(\lambda_2))\quad \lambda_1,\lambda_2\in \DD.$$ A direct consequence of this fact is that $\Psi$ preserves the royal variety of the symmetrised bidisc, i.e. $\Psi(\Sigma)=\Sigma$.

Note that a point $(a,s,p)\in \overline\DD \times (\Sigma\cap \partial \GG_2)$ lies in $\bar\PP$ if and only if $a=0$. This, in particular, means that $\lambda \mapsto \varphi_1(0,2\lambda, \lambda^2)$ vanishes on $\TT$, whence $(s,p)\mapsto \varphi_1(0,s,p)$ vanishes on $\Sigma$. Therefore, 
\begin{equation}\label{a} 
\varphi(0, 2\lambda, \lambda^2) = (0, 2 b(\lambda), b(\lambda)^2),\quad \lambda\in \DD.
\end{equation} 
Moreover, there is a holomorphic map $\alpha_1$ on $\mathbb G_2$ such that $\varphi_1(0,s,p)= (s^2 - 4p) \alpha_1(s,p)$ for $(s,p)\in \GG_2$.

Step~3. Applying \eqref{a} to $\varphi^{-1}$ we get immediately that $b$ is a M\"obius function. Composing $\varphi$ with an automorphism of the form \eqref{con} we may assume that $b(\lambda)=\lambda$, $\lambda\in \DD$.

For now we have shown that, up to a composition with an automorphism of the form \eqref{con}, $$\varphi(0,s,p)=((s^2-4p)\alpha_1(s,p), s,p),\quad (s,p)\in \GG_2.$$
In particular,
\begin{multline}\label{mult:phi}
\varphi(a,s,p)=( (s^2-4p)\alpha_1(s,p) +\alpha_2 a+ a\alpha_3(a,s,p),\\ s+ a\beta_1 + a\beta_2(a,s,p), p+ a \gamma_1 + a\gamma_2(a,s,p))\end{multline} 
for $(a,s,p)\in \PP$, where $\alpha_2,\beta_1,\gamma_1\in \CC$ and $\alpha_3,\beta_2,\gamma_2\in \OO(\PP)$ vanish at the origin.

Step~4. Now we shall make use of the fact that the pentablock is $(2,1,2)$- and $(1,1,2)$-balanced. For $m=(m_1,m_2,m_3)\in \NN^3$ and $\lambda \in \CC$ denote the action on $\CC^3$ $$m_\lambda.x=(\lambda^{m_1} x_1, \lambda^{m_2} x_2,\lambda^{m_3} x_3).$$

Let $m=(2,1,2)$. It follows from Section~\ref{sec:act} that for any $x\in \PP$ and $\lambda\in \DD$ the point $m_\lambda.x$ lies in $\PP$. For $\lambda$ in the pointed disc let us define $$\phi_\lambda: (a,s,p)\mapsto m_{\lambda^{-1}}.\phi(m_\lambda.(a,s,p)).$$ Note that for any unimodular $\lambda$ the mapping $\phi_\lambda$ is an automorphism of $\mathcal P$. Moreover, $(\phi_\lambda)^{-1} = (\phi^{-1})_\lambda$, $\lambda\in \TT$. 

For a fixed $(a,s,p)\in\mathcal P$ the mappings $\lambda\mapsto \varphi_\lambda(a,s,p)$ and $\lambda\mapsto \varphi^{-1}_\lambda (a,s,p)$ are analytic on $\overline\DD\setminus \{0\}$. Moreover, making use of formula \eqref{mult:phi} (which may be applied to $\varphi^{-1}$ as well) we find that these mappings extend holomorphically past $0$, let say to $\varphi_0$ and $\varphi_0^{-1}$ respectively. Clearly, the equalities $\varphi_\lambda \circ \varphi_\lambda^{-1}=\id$ and $\varphi^{-1}_\lambda \circ \varphi_\lambda=\id$ hold on $\TT$, so thanks to the analicity with respect to $\lambda$ they remain true on the whole disc.

Thus, for a fixed $(a,s,p)\in \PP$ the mappings $\lambda\mapsto \varphi_\lambda(a,s,p)$ and $\lambda\mapsto \varphi^{-1}_\lambda (a,s,p)$ are analytic on a neighborhood of $\overline \DD$ and they map $\TT$ into the pentablock. From this we deduce that they map $\DD$ into the pentablock (it is a direct consequence of the polynomial convexity of $\PP$ -- see \cite[Theorem~6.3]{Agl-Lyk-You14}).

In particular, letting $\lambda\to 0$ we find that $\varphi_0$ and $\varphi_0^{-1}$ are automorphisms of the pentablock.

Using of \eqref{mult:phi} one may compute the formula for $\varphi_0$ to deduce that the mapping $$(a,s,p)\mapsto ((s^2-4p)\alpha_1(0,0) + \alpha_2 a, s, \gamma_1 a +p)$$ is an automorphism of the pentablock. Putting $s=0$ and making use of the description of the pentablock (note that $\PP\cap (\CC\times \{0\}\times \CC) =\DD\times \{0\} \times \DD$, by Theorem~\ref{th:severaldescriptions} (2)) we infer that $$(a,p)\mapsto (-4p\alpha_1(0,0) + \alpha_2 a, \gamma_1 a +p)$$ is an automorphism of the bidisc. Therefore $\gamma_1=\alpha_1(0,0)=0$ and $|\alpha_2|=1$. Losing no generality we may assume that $\alpha_2=1$ (compose, if necessary, $\varphi$ with $(a,s,p)\mapsto (\alpha_2^{-1} a, s , p)$).

Applying the same reasoning to $m=(1,1,2)$ (it is possible as $\gamma_1=0$) we find that the following mapping $$(a,s,p)\mapsto (a, \beta_1 a +s, p+\gamma_3 as + \gamma_4 a^2)$$ is an automorphism of $\PP$ for some $\gamma_3,\gamma_4\in \CC$. In particular, putting $s=0$ we see that for any $a,p\in \DD$ the point $(a,\beta_1 a, p+ \gamma_4 a^2)$ lies in the pentablock. In particular, $(1,\beta_1 ,p + \gamma_4)$ lies in $\bar\PP$ for any $p\in \DD$. This means, that there is $z\in \bar{\mathcal R}_I$ such that $\pi(z)=(1,\beta_1,p + \gamma_4)$. Since $z_{12}=1$ we see that $z_{11}=z_{22}=0$, whence $\beta_1=0$.

Step~5. Note that we have shown that $\varphi'(0)=\id$. Since $\varphi$ preserves the origin, the assertion is a direct consequence of the classical Cartan theorem.

The second part of the theorem is clear.
\end{proof}

\section{Retracts of $\mathcal J^*$ algebras}

Finally we shall prove that the pentablock is not an analytic retract of the open unit ball of any $\mathcal J^*$ algebra of finite rank (see \cite{Har} for a definition of a $\mathcal J^*$ algebra). A domain $P$ is said to be an \emph{analytic retract} of a domain $B$ if there are holomorphic mappings $\phi:P\to B$ and $\psi:B\to P$ such that $\psi\circ \phi = \id$.
  
Recall a result of N.~Young who showed in \cite{You 2007} that the tetrablock (another domain appearing in the $\mu$-synthesis problem) possesses this property. Some analysis of his proof shows that it works for the symmetrised bidisc as well. Actually, the main ingredient of Young's proof is to show that there is no isometry between the unit ball in $\CC^3$ with the norm $||x||_{\mathbb E}:=\max(|x_1|,|x_2|) + |x_3|$ and the mentioned unit ball in the $\mathcal J^*$ algebra of finite rank. However, the argument used there (and this is actually the core of the idea) shows that such an isometry does not exist if $\CC^3$ equipped with $||\cdot||_{\mathbb E}$ will be replaced with $\CC^2$ equipped with the norm $||x||_s=|x_1|+|x_2|$. On the other hand the indicatrix of the symmetrised bidisc at the origin is clearly isomorphic with the unit ball in $(\CC^2, ||\cdot||_s)$.

Using this we may simply prove that the same property holds for $\PP$:
\begin{proposition}\label{retrp}
The pentablock is not an analytic retract of the open unit ball of a $\mathcal J^*$ algebra of finite rank.
\end{proposition}

\begin{proof}
Suppose a contrary, i.e. there is a $\mathcal J^*$ algebra $\mathcal A$ and there are analytic mappings $\varphi:\PP \to \mathcal B$ and $\psi:\mathcal B\to \PP$ such that $\psi\circ \phi = \id,$ where $\mathcal B$ is the unit ball in $\mathcal A$.

Let us define $$i:\GG_2\ni (s,p)\mapsto (0,s,p)\in \mathcal P$$ and $$j:\mathcal P\ni (a,s,p)\mapsto (s,p)\in \GG_2.$$ Then $\tilde \phi:=\phi\circ i:\GG_2 \to \mathcal B$ and $\tilde \psi := j\circ \psi: \mathcal B \to \GG_2$ satisfy $\tilde \psi\circ \tilde \phi=\id$; a contradiction.
\end{proof}

\begin{remark}
Of course, the same argument works for any complete Hartogs domain whose base is $\GG_2$ or, more generally, whose base is any other domain satisfying the assertion of Proposition~\ref{retrp}.
\end{remark}

\bigskip

The author would like to thank Piotr Jucha for some helpful conversations. He is also very grateful to Pawe{\l} Zapa{\l}owski for reading a draft version of the manuscript and number of remarks.





\end{document}